\documentclass[12pt]{amsart}
\usepackage[a4paper,margin=3cm]{geometry}
\usepackage{amsmath,amssymb,amsthm}
\usepackage[T1]{fontenc}
\usepackage{newtxtext,newtxmath}
\usepackage[colorlinks=true,linkcolor=blue,citecolor=blue,urlcolor=blue]{hyperref}
\usepackage{tikz}
\newtheorem{theorem}{Theorem}
\newtheorem{lemma}[theorem]{Lemma}

\newtheorem{observation}[theorem]{Observation}
\theoremstyle{remark}
\newtheorem{remark}[theorem]{Remark}

\newcommand{\N}{\mathcal N}
\newcommand{\Ske}{\mathcal S}

\usetikzlibrary{backgrounds}
\usetikzlibrary{arrows}
\pgfdeclarelayer{edgelayer}
\pgfdeclarelayer{nodelayer}
\pgfsetlayers{background,edgelayer,nodelayer,main}
\tikzstyle{none}=[inner sep=0mm]
\tikzstyle{bluenode}=[fill=blue, draw=black, shape=circle, minimum
size=0.2cm,
inner sep=0pt]
\tikzstyle{whitenode}=[fill={rgb,255: red,245; green,245;
blue,245}, draw=black, shape=circle, minimum size=0.1cm, inner
sep=1pt]
\tikzstyle{yellownode}=[fill=yellow, draw=black, shape=circle, minimum size=0cm, inner sep=1pt]
\tikzstyle{pinknode}=[fill={rgb,255: red,255; green,191; blue,191}, draw=black, shape=circle, minimum size=0cm, inner sep=1pt]
\tikzstyle{blacknode}=[fill=black, draw=black, shape=circle, minimum size=0.2cm, inner sep=0pt]
\tikzstyle{rednode}=[fill={rgb,255: red,244; green,0; blue,0}, draw=black,
shape=circle, minimum size=0.2cm, inner sep=0.1pt]
\tikzstyle{square}=[draw=black, shape=rectangle, minimum
size=0.1cm, inner sep=3pt,fill={rgb,255: red,245; green,245;
blue,245},scale=0.5]
\tikzstyle{dot}=[fill=black, draw=black, shape=circle, minimum size=0.04cm, inner sep=0pt]
\tikzstyle{blackedge}=[line width=1.2pt, black]
\tikzstyle{blackedge_thick}=[-, draw=black, line width=1.5pt,
fill=none]
\tikzstyle{rededge}=[-, draw=red]
\tikzstyle{rededge_thick}=[-, line width=0.45mm, draw=red]
\tikzstyle{blackedge_opacity}=[-, -, draw={rgb,255: red,91; green,87; blue,84},
fill=none, line width=0.4mm, opacity=0.7]
\tikzstyle{balck_dash}=[-, dash pattern=on 0.2mm off 0.2mm]
\tikzstyle{blue_thick}=[-, line width=0.5mm, draw=blue]
\tikzstyle{blueedge}=[-, line width=1pt, blue, opacity=0.7]

\title{The maximum number of $k$-Cliques of 7-Connected 1-Planar Graphs}

\author{Yuanqiu Huang}
\address{School of Mathematics and Statistics, Hunan Normal University,
Changsha, Hunan 410081, P.R. China}
\email{hyqq@hunnu.edu.cn}

\author{Licheng Zhang}
\address{School of Mathematics and Statistics, Hunan Normal University,
Changsha, Hunan 410081, P.R. China}
\email{lczhangmath@163.com}

\thanks{The work was supported by the National Natural Science Foundation of China
(Grant Nos. 12271157 and 12371346).}
\thanks{Corresponding author: Licheng Zhang.}

\begin{document}

\begin{abstract}
In 2023, Gollin, Hendrey, Methuku, Tompkins and Zhang determined the maximum number of
cliques in general \(1\)-planar graphs with order $n$. Their extremal examples have
connectivity at most three, except for a few small orders. At the high-connectivity end, we prove that every \(n\)-vertex \(7\)-connected \(1\)-planar graph has at most \(4n-12\) edges, \(4n-16\) triangles, and \(n-6\) copies of \(K_4\). Hence the total number of cliques is at most \(10n-33\). All bounds are sharp for infinitely many values of \(n\).
\end{abstract}
\maketitle

\section{Introduction}
A \(k\)-clique of a graph is a set of \(k\) vertices that induces a copy of the complete graph \(K_k\).
Thus extremal edge problems are precisely \(2\)-clique-counting problems.
Zykov~\cite{Zykov} extended Tur\'an's theorem in this direction by determining
the maximum number of small cliques while forbidding larger cliques.
The problem of determining the maximum number of \(k\)-cliques has also been well studied for planar graphs. It is konwn that every $n$-vertex ($n\ge 3$) planar graph has at most $3n-6$ edges. Hakimi and Schmeichel~\cite{Hakimi} proved that an $n$-vertex planar graph contains at most $3n-8$ triangles, while Alon and Caro~\cite{Alon}, and independently Wood~\cite{Wood}, proved that it contains at most $n-3$ copies of $K_4$. The extremal graphs for both bounds are Apollonian networks. In contrast, if a planar graph has no separating triangle, then all its triangles are facial and it has no copy of $K_4$; in particular, a 4-connected $n$-vertex planar graph has at most $2n-4$ triangles and, for $n\ge 5$, no copy of $K_4$.

Introduced by Ringel in connection with the six-colour problem~\cite{Ringel},
\(1\)-planar graphs form a natural generalization of planar graphs.
A graph is \emph{1-planar} if it admits a drawing in the plane in which each edge is crossed at most once. A \emph{1-plane graph} is a 1-planar graph together with a fixed such drawing.  In 2023,  Gollin, Hendrey, Methuku, Tompkins and Zhang~\cite{Gollin} gave a complete
solution to the clique-counting problem for general \(1\)-planar graphs,
determining the sharp bounds for every fixed clique size and for the total
number of cliques. More precisely, if \(n=3q+s\) with \(s\in\{0,1,2\}\), then, apart from the
exceptional case \(n=8\), where the maximum number of triangles is \(32\), they
proved that the maximum number of triangles is
$
        19q+5s-18,
$
and, for \(4\le t\le 6\), the maximum number of copies of \(K_t\) is
$
        (q-1)\binom 6t+\binom{s+3}{t}.
$
They also determined the maximum total number of cliques, which is
$
        56(q-1)+2^{s+3}.
$
Furthermore, they gave a characterization of the extremal graphs. One consequence of that characterization is that,
except for a few small orders, the extremal graphs have vertex-connectivity at most three.

For \(1\)-planar graphs, the relevant connectivity range is quite small.
Every \(1\)-planar graph has a vertex of degree at most seven~\cite{Fabrici},
so no \(1\)-planar graph can be more than \(7\)-connected.  Thus, once one asks
for clique counts under connectivity assumptions, the natural cases are
\(r\)-connected \(1\)-planar graphs with \(4\le r\le 7\). Some general bounds follow directly from Corollary~5.3 of
Gollin et al.~\cite{Gollin}.  Indeed, let \(G\) be a \(4\)-connected
\(1\)-planar graph of order \(n\ge 7\), and extend \(G\) on the same vertex set
to a maximal \(1\)-planar graph \(G^+\). Then \(G^+\) is still
\(4\)-connected, and is \(K_6\)-free by Lemma~6.1 of
Gollin et al.~\cite{Gollin}. Hence Corollary~5.3 applies to \(G^+\), and since
\(G\subseteq G^+\), it gives
\[
        \N(G,K_3)\le 6n-14,\qquad
        \N(G,K_4)\le 4n-9,\qquad
        \N(G,K_5)\le n-2 .
\]
However, there is currently no evidence that these bounds are sharp for any
fixed connectivity \(k\) with \(4\le k\le 7\).

In this note we treat the endpoint \(r=7\).   For a graph \(G\), let \(\N(G,K_t)\) denote the number of copies of \(K_t\) in
\(G\).  We count the empty set as a clique; equivalently,
$
        \N(G,K_0)=1.
$

\begin{theorem}\label{thm:main}
Let \(G\) be a \(7\)-connected \(1\)-planar graph of order \(n\). Then
\[
\mathcal N(G,K_t)\le
\begin{cases}
4n-12, & t=2,\\
4n-16, & t=3,\\
n-6, & t=4,\\
0, & t\ge 5.
\end{cases}
\]
Consequently,  the total number
of cliques in \(G\) is at most \(10n-33\).  All these bounds are sharp for
infinitely many values of \(n\).
\end{theorem}

The proof has two ingredients. First, in a rich 1-planar drawing of a maximal 7-connected 1-planar graph, the planar skeleton is 4-connected. This removes the separator term in the triangle-counting formula of Gollin et al. Second, Biedl's lower bound on the number of uncrossed triangular faces in a triangulated 1-plane graph of minimum degree  seven gives the exact deficit from the optimal 1-planar edge bound $4n-8$.

\section{Preliminaries}
Unless stated otherwise, all graphs considered in this paper are finite and
simple.  For a graph \(G\), let \(V(G)\) and \(E(G)\) denote its vertex set and
edge set, respectively. For a vertex set \(S\subseteq V(G)\), we write \(G-S\) for
the graph obtained from \(G\) by deleting the vertices in \(S\) and all edges
incident with them. A vertex set \(S\subseteq V(H)\) is a separator of \(H\) if
\(G-S\) is disconnected; it is a \(k\)-separator if \(|S|=k\).
 Let $G$ be a 1-plane graph. 
A \emph{face} of $G$ is a connected component of $\mathbb{R}^2 \setminus G$.   A face of $G$ is \emph{crossed} if a crossing point lies on its boundary, and \emph{uncrossed} otherwise. If the boundary of an uncrossed face is a cycle, we call this cycle a \emph{facial cycle}.For an edge $e$ of $G$, the crossing points lying on $e$, together with the two endvertices of $e$, divide $e$ into arcs, called the \emph{edge-segments} of $e$. The \emph{degree} $\deg(f)$ of a face $f$ is the number of edge-segments encountered in a closed walk along the boundary of $f$, counted with multiplicity. A face of degree \(k\) is called a \emph{\(k\)-face}. In particular, a \(3\)-face is also
called a triangular face. For a 1-planar drawing $\phi$ of $G$, the \emph{planar skeleton} $\Ske(\phi)$ is the plane spanning subgraph consisting of the uncrossed edges. If $H$ is a plane graph, write $f_i(H)$ for the number of faces of degree $i$, and write $s_3(H)$ for the number of 3-separators of $H$.   A \(1\)-planar graph \(G\) is called \emph{maximal \(1\)-planar} if no edge
can be added between two non-adjacent vertices of \(G\) so that the resulting
graph is still \(1\)-planar.

A cycle containing exactly \(k\) vertices is called a \(k\)-cycle.
In particular, a \(3\)-cycle is also called a triangle.

We shall use the following elementary planar separation observation.

\begin{observation}\label{lem:separator-curve}

Let \(H\) be a \(3\)-connected plane graph, and let \(S\) be a
\(3\)-separator of \(H\). Then there exists a simple closed curve
\(\ell\) in the plane such that \(\ell\cap H=S\), and each component of
the plane minus \(\ell\) contains at least one component of \(H-S\).
\end{observation}

%
%

The following result of Biedl  expresses the edge deficit of a triangulated 1-plane graph in terms of its uncrossed triangular faces.

\begin{lemma}[Biedl~\cite{Biedl}]\label{lem:biedl}
Let $G$ be a triangulated 1-plane multigraph, and let $t_3(G)$ be the number of its uncrossed triangular faces. Then
\[
|E(G)|=4|V(G)|-8-\frac{t_3(G)}2.
\]
If, in addition, $\delta(G)= 7$, then $t_3(G)\ge 8$.
\end{lemma}

Let \(C\) be a cycle in a \(1\)-plane graph \(G\) such that no two edges of
\(C\) cross each other. Thus, the closed curve induced by the edges of \(C\)
separates the plane into two regions \(C_{\operatorname{int}}\) and
\(C_{\operatorname{out}}\). We call \(C\) a \emph{conflict cycle} if both
regions \(C_{\operatorname{int}}\) and \(C_{\operatorname{out}}\) contain one
or more vertices of \(G-V(C)\).

\begin{lemma}[Huang--Zhang--Wang~\cite{Huang}]\label{lem:no-sep-triangle}
No 3-cycle in a 7-connected 1-plane graph is conflict.
\end{lemma}

Let \(\mathbb{S}^2\) denote the sphere. Two \(1\)-planar drawings
\(\phi,\phi' : G \to \mathbb{S}^2\) are said to be \emph{weakly equivalent}
if there exist an automorphism \(\sigma\) of \(G\) and a self-homeomorphism
\(h\) of \(\mathbb{S}^2\) such that
$
h\circ \phi=\phi'\circ \sigma .
$

\begin{lemma}\label{lem:no-large-clique}
If \(G\) is a \(6\)-connected \(1\)-planar graph, then $G$ contains no $K_5$.
\end{lemma}

\begin{proof}

Suppose, to the contrary, that \(G\) contains a subgraph $H$ isomorphic to $K_5$.
 Fix a \(1\)-planar drawing \(\varphi\) of \(G\), and let
\(\varphi|_H\) denote the restriction of \(\varphi\) to \(H\). The \(1\)-planar drawing of \(K_5\) is unique up to weak 
equivalence \cite{MatsumotoSuzuki}; see Figure~\ref{fig:K4K5-unique-1planar-drawings}(c).
Since \(G\) is \(6\)-connected, we have \(|V(G)|\ge 7\), and hence
\(G\) has a vertex outside \(H\). We shall show that no vertex
\(v\in V(G)\setminus V(H)\) can lie in any face of \(\varphi|_H\).

First, \(v\) cannot lie in a crossed triangular face  of \(\varphi|_H\). Indeed, if \(v\) lies
in such a face $T$, then for the vertex of \(H\) not incident with $T$,
there are at most three internally vertex-disjoint paths from \(v\) to that
vertex. By Menger's theorem, this implies that \(\kappa(G)\le 3\), a
contradiction.

Thus \(v\) must lie in an uncrossed triangular region of \(\varphi|_H\).
By the structure of the unique \(1\)-plane drawing of \(K_5\), each uncrossed
triangular region is adjacent, along one of its boundary edges, to a crossed
triangular region. Without loss of generality, suppose that \(v\) lies in the
uncrossed triangular face bounded by \(u_1u_2u_5u_1\), and that the adjacent
crossed triangular face along \(u_1u_2\) is bounded by \(u_1u_2xu_1\), where
\(x\) is the crossing point. Then there are at most five internally
vertex-disjoint paths from \(v\) to \(u_4\). Again by Menger's theorem,
\(\kappa(G)\le 5\), contradicting the \(6\)-connectivity of \(G\).

\end{proof}

\begin{figure}
\centering
\begin{tikzpicture}
    \tikzset{
        whitenode/.style={circle, draw, fill=black!10, inner sep=2pt},
        blueedge/.style={draw=blue, line width=1pt, opacity=0.7},
        blackedge/.style={draw=black, line width=1pt}
    }

    \def\xone{0}
    \def\xtwo{4.2}
    \def\xthree{8.4}

    \begin{scope}[shift={(\xone,0)}]
        \node[whitenode] (A) at (0,0.75) {};
        \node[whitenode] (B) at (0,2.05) {};
        \node[whitenode] (C) at (-1.13,0) {};
        \node[whitenode] (D) at (1.13,0) {};

        \draw[blackedge] (A) -- (B) -- (C) -- (A);
        \draw[blackedge] (A) -- (D) -- (B);
        \draw[blackedge] (C) -- (D);

        \node at (0,-0.75) {$(a)$};
    \end{scope}

    \begin{scope}[shift={(\xtwo,0)}]
        \node[whitenode] (E) at (-1,2) {};
        \node[whitenode] (F) at (1,2) {};
        \node[whitenode] (G) at (-1,0) {};
        \node[whitenode] (H) at (1,0) {};

        \draw[blackedge] (E) -- (F);
        \draw[blackedge] (G) -- (H);
        \draw[blackedge] (E) -- (G);
        \draw[blackedge] (F) -- (H);
        \draw[blueedge] (E) -- (H);
        \draw[blueedge] (F) -- (G);

        \node at (0,-0.75) {$(b)$};
    \end{scope}

    \begin{scope}[shift={(\xthree,0)}]
        \node[whitenode,label=above:$u_1$] (u1) at (-1,2) {};
        \node[whitenode,label=above:$u_2$] (u2) at (1,2) {};
        \node[whitenode,label=below:$u_3$] (u3) at (-1,0) {};
        \node[whitenode,label=below:$u_4$] (u4) at (1,0) {};
        \node[whitenode,label=above:$u_5$] (u5) at (0,3.1) {};
           \node[label=above:$x$] (x) at (0,1) {};
        \draw[blackedge] (u2)--(u5)--(u1)--(u2)--(u4)--(u3)--(u1);
        \draw[blueedge] (u1) -- (u4);
        \draw[blueedge] (u2) -- (u3);

        \draw[blackedge] (u5) to[in=120, out=-165, looseness=1.4] (u3);
        \draw[blackedge] (u5) to[in=60, out=-15, looseness=1.4] (u4);

        \node at (0,-0.75) {$(c)$};
    \end{scope}

\end{tikzpicture}
\caption{The two weakly inequivalent \(1\)-planar drawings of \(K_4\): 
(a) tetrahedral and (b) pyramidal; and (c) the unique \(1\)-planar drawing of \(K_5\).}
\label{fig:K4K5-unique-1planar-drawings}
\end{figure}
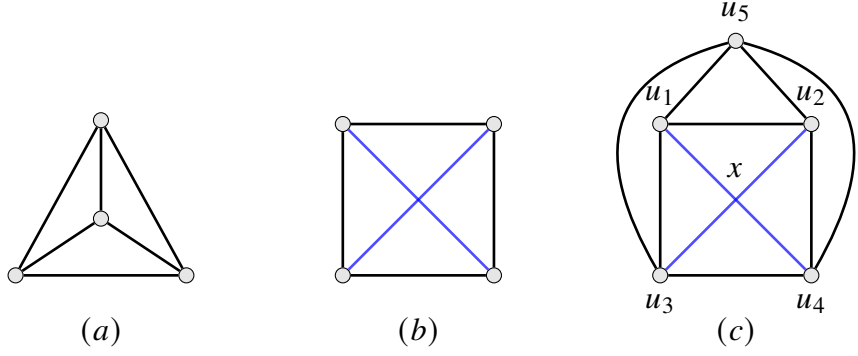

A 1-planar drawing is called \emph{rich} if every crossing pair forms a kite: whenever $vw$ and $xy$ cross, the four endpoints $v,w,x,y$ induce a $K_4$, and the four edges other than $vw$ and $xy$ are uncrossed. Gollin et al.~\cite{Gollin} proved the following elegant structural lemma.  In their original statement the last formula is given as an upper
bound; however, it is not difficult to see that equality
actually holds.
\begin{lemma}[Gollin et al.~\cite{Gollin}]\label{lem:gollin}
Let $G$ be a 3-connected maximal 1-planar graph of order at least five. Then $G$ has a rich 1-planar drawing $\phi$ such that $H=\Ske(\phi)$ is 3-connected and every face of $H$ has degree three or four. Moreover,
\[
\N(G,K_3)=s_3(H)+f_3(H)+4f_4(H).
\]
\end{lemma}


\section{The planar skeleton in the 7-connected case}
We next show that, at maximum connectivity, the skeleton has no 3-separator. This is the point where the endpoint assumption is used in an essential way.

\begin{lemma}\label{lem:pyramidal-k4}
Let $G$ be a 7-connected 1-plane graph.Then every copy of $K_4$ in $G$ is shown  Figure~\ref{fig:K4K5-unique-1planar-drawings}(b).
\end{lemma}

\begin{proof}
There are two weakly inequivalent \(1\)-planar drawings of \(K_4\): the planar
tetrahedral drawing and the pyramidal drawing with one crossing
\cite{MatsumotoSuzuki}; see Figure~\ref{fig:K4K5-unique-1planar-drawings}(a)
and (b). Suppose that a copy \(Q\) of \(K_4\) in \(G\) is drawn as a planar
tetrahedron. Since \(G\) is \(7\)-connected, \(|V(G)|\ge 8\), and hence there
is a vertex \(z\in V(G)\setminus V(Q)\). In the drawing restricted to \(Q\),
the vertex \(z\) lies in a triangular face of the tetrahedral drawing.  It follows that the boundary of  the triangular face  bounds a conflicting \(3\)-cycle of
\(G\), contradicting Lemma~\ref{lem:no-sep-triangle}.
\end{proof}

In our previous joint work with Wang~\cite{Huang}, we proved that every
\(7\)-connected maximal \(1\)-planar graph contains a spanning \(4\)-connected
plane graph obtained by deleting one edge from each crossing pair.  Here we
need a stronger form: in a rich 1-planar drawing, the  planar skeleton  is already
\(4\)-connected.

\begin{lemma}\label{lem:skeleton-four-connected}
Let \(G\) be a \(7\)-connected maximal \(1\)-planar graph, and let \(\phi\)
be a rich \(1\)-planar drawing of \(G\). Then \(\Ske(\phi)\) is
\(4\)-connected.
\end{lemma}

\begin{proof}
Put \(H=\Ske(\phi)\). By Lemma~\ref{lem:gollin}, \(H\) is \(3\)-connected and
every face of \(H\) has degree three or four. Suppose, to the contrary, that
\(H\) has a \(3\)-separator \(S=\{x,y,z\}\).

By Observation~\ref{lem:separator-curve}, there is a simple closed curve \(\ell\)
meeting \(H\) exactly in the vertices \(x,y,z\), such that the two regions
separated by \(\ell\) both contain vertices of \(H-S\). We may assume that
\(x,y,z\) appear on \(\ell\) in this cyclic order. Let \(\ell_{xy}\),
\(\ell_{yz}\), and \(\ell_{zx}\) denote the three subarcs of \(\ell\) between
consecutive vertices of \(S\).

The vertices \(x,y,z\) cannot all lie on the boundary of a single face of
\(H\). Indeed, in a \(3\)-connected plane graph, the vertex set of a facial
cycle is not a separator~\cite{Tutte}.

Now consider \(\ell_{xy}\), and let \(F_{xy}\) be the face of \(H\) containing
it. Since every face of \(H\) has degree three or four, either \(x\) and \(y\)
are adjacent on the boundary of \(F_{xy}\), or \(F_{xy}\) is a \(4\)-face and
\(x,y\) are opposite vertices of \(F_{xy}\). In the first case, \(xy\in E(H)\).
In the second case, the richness of \(\phi\) gives the crossed diagonal
\(xy\in E(G)\) inside \(F_{xy}\). Hence \(xy\in E(G)\). The same argument
applied to \(\ell_{yz}\) and \(\ell_{zx}\) gives \(yz,zx\in E(G)\).

Therefore \(xyzx\) is a \(3\)-cycle of \(G\). Moreover, replacing each arc
\(\ell_{xy},\ell_{yz},\ell_{zx}\) by the corresponding edge \(xy,yz,zx\) is
done inside the same face of \(H\), and hence does not move any vertex of
\(H-S\) from one side of the curve to the other. Thus the closed curve induced
by the \(3\)-cycle \(xyzx\) still has vertices of \(H-S\) on both sides. Hence
\(xyzx\) is a conflict \(3\)-cycle in \(\phi\), contradicting
Lemma~\ref{lem:no-sep-triangle}. Therefore \(H\) is \(4\)-connected.
\end{proof}

\begin{lemma}\label{lem:counts-by-faces}
Let $G$ be a 7-connected maximal 1-planar graph, let $\phi$ be a rich 1-planar drawing of $G$, and put $H=\Ske(\phi)$. Then
\[
\N(G,K_3)=f_3(H)+4f_4(H),
\qquad
\N(G,K_4)=f_4(H).
\]
\end{lemma}

\begin{proof}
By Lemma~\ref{lem:skeleton-four-connected}, $s_3(H)=0$. The formula for triangles follows immediately from Lemma~\ref{lem:gollin}.
By richness, every 4-face of \(H\) gives a copy of \(K_4\) in \(G\). Conversely,
Lemma~\ref{lem:pyramidal-k4} implies that every copy of \(K_4\) in \(G\) is
pyramidal, and by Lemma~\ref{lem:gollin} all \(K_4\) comes from a 4-face of \(H\). Hence $\N(G,K_4)=f_4(H)$.

\end{proof}

\section{Proof of the theorem}
We first prove the upper bounds. Let $G$ be a 7-connected 1-planar graph. We may add edges to a fixed 1-planar drawing until it becomes maximal 1-planar on the same vertex set. Adding edges cannot decrease vertex-connectivity and cannot decrease the number of cliques. Thus it suffices to prove the bounds for a maximal 1-planar supergraph of $G$.

Assume, then, that $G$ is maximal 1-planar, and choose a rich 1-planar drawing $\phi$ given by Lemma~\ref{lem:gollin}. Put $H=\Ske(\phi)$ and write $f_i=f_i(H)$. Since every face of $H$ has degree three or four, Euler's formula gives
\begin{equation}\label{eq:euler-faces}
        f_3+2f_4=2n-4.
\end{equation}
Indeed, from $3f_3+4f_4=2|E(H)|$ and $n-|E(H)|+f_3+f_4=2$, eliminating $|E(H)|$ gives \eqref{eq:euler-faces}.

The uncrossed triangular faces of the rich 1-planar drawing are exactly the 3-faces of $H$. Since $G$ is 7-connected and 1-planar, $\delta(G)= 7$. Applying Lemma~\ref{lem:biedl} to the (triangulated) rich 1-planar  drawing gives
\begin{equation}\label{eq:f3-lower}
        f_3\ge 8.
\end{equation}

The edge bound follows from the edge formula in Lemma~\ref{lem:biedl}:
\[
|E(G)|=4n-8-\frac{f_3}{2}\le 4n-12.
\]
Thus $\N(G,K_2)\le 4n-12$.

For triangles, Lemma~\ref{lem:counts-by-faces} and \eqref{eq:euler-faces} give
\[
\N(G,K_3)=f_3+4f_4
       =f_3+2(2n-4-f_3)
       =4n-8-f_3.
\]
Using \eqref{eq:f3-lower}, we get
\[
        \N(G,K_3)\le 4n-16.
\]
Similarly,
\[
\N(G,K_4)=f_4=\frac{2n-4-f_3}{2}\le n-6.
\]
Finally, Lemma~\ref{lem:no-large-clique} gives $\N(G,K_t)=0$ for every $t\ge 5$. The cases $t=0$ and $t=1$ are immediate from $\N(G,K_0)=1$ and from $\N(G,K_1)=n$.

Therefore
\[
\sum_{t\ge 0}\N(G,K_t)
\le 1+n+(4n-12)+(4n-16)+(n-6)=10n-33.
\]
This proves the asserted upper bound on the total number of cliques.

The bounds are attained for infinitely many orders.  Hoffmann, Reddy and
Seemann~\cite{HoffmannReddySeemann} constructed \(7\)-connected triangulated
1-planar graphs \(G_k\) of order \(n=24+8k\) where \(k\) is a nonnegative integer, called
\(k\)-layered double stop-sign graphs.  One checks directly that these graphs satisfy the required conditions. Thus all
bounds in Theorem~\ref{thm:main} are sharp for infinitely many values of \(n\).

\begin{remark}It is not hard to see from the proof that, unlike in the general setting
considered by Gollin et al.~\cite{Gollin}, the same extremal examples attain
the bounds for all clique sizes considered here in the \(7\)-connected case.
\end{remark}
\section{Further questions}

The examples of Hoffmann et al. \cite{HoffmannReddySeemann} show that equality in
Theorem~\ref{thm:main} is attained for infinitely many orders
\(n=24+8k\).  It would be very interesting to determine the remaining orders
for which equality can be attained.  Orders \(25\) and \(27\) are already
exceptional, since there is no \(7\)-connected \(1\)-planar graph of these
orders \cite{Huang}.  Apart from the exceptional orders \(25\) and \(27\),
we believe that the upper bounds are attainable for all remaining orders $n\ge 24$.

Furthermore, it remains to determine the sharp upper bounds for
\(\mathcal N(G,K_t)\) in  \(1\)-planar graphs with connectivity $r$ for
\(r=4,5,6\) and \(3\le t\le 5\).

\end{document}